\documentclass[15pt]{amsart}

\usepackage[centertags]{amsmath}
\usepackage{amsfonts}
\usepackage{amsthm}
\usepackage{newlfont}
\usepackage{amscd}
\usepackage{amsgen}
\usepackage{amssymb}
\usepackage{longtable}
\usepackage[centertags]{amsmath}
\usepackage{amsfonts}
\usepackage{amsthm}
\usepackage{newlfont}
\usepackage{amscd}
\usepackage{amsgen}
\usepackage{amssymb}

\newlength{\defbaselineskip} 

 \theoremstyle{plain} \newtheorem{thm}{Theorem}[section]
 \newtheorem{lemm}[thm]{Lemma}
\newtheorem{prop}[thm]{Proposition}
\theoremstyle{definition}

\newtheorem{cor}[thm]{Corrolary}
\newtheorem{rem}[thm]{Remark}

 \numberwithin{equation}{section}
\numberwithin{equation}{section} \theoremstyle{definition}
\newtheorem{ex}{Example}[section]

\DeclareMathOperator{\sing}{sing}

\begin{document}
\title{Projections of del Pezzo surfaces and Calabi--Yau
threefolds}
\author{Grzegorz Kapustka}

\thanks{
Mathematics Subject Classification(2000): 14J10, 14J15, 14J30,
14J32.\\
Partially supported by MNiSW grant N N201 388834. The author is
supported by the Foundation for Polish Science (FNP)} \maketitle

 \begin{abstract} We study the syzygetic structure of projections of del Pezzo surfaces in order to
 construct singular Calabi--Yau threefolds. By smoothing those threefolds, we obtain
 new examples of Calabi--Yau threefolds with Picard
 group of rank $1$. We also give an example of type II primitive
contraction whose exceptional divisor is the blow-up of the
projective plane at a point.
 \end{abstract}
\maketitle
\section{Introduction}
A Calabi--Yau threefold is an irreducible projective threefold with
Gorenstein singularities, trivial canonical class such that the
intermediate cohomologies of its structure sheaf are all trivial
($h^i(X,\mathcal{O}_X) = 0$ for $0 < i < \dim(X)$). The aim of this
work is to extend and complement the results obtained in \cite{K}
and \cite{KK}. Considering projections of del Pezzo surfaces, we
find new examples of smooth Calabi--Yau threefolds with Picard group
of rank $1$. Note that these threefolds of low degree are not
complete intersection in toric varieties (see \cite{BB}) and are
interesting from the point of view of the Picard--Fuchs equations
and mirror symmetry (see \cite{BCKS}, \cite{CC}, \cite{CD}, \cite{COGP}, \cite{ED} and Remark \ref{re}). In
particular it is an open problem to find their mirror families (or at leat a candidate for the Picard-Fuchs equation of
the mirror). Recall that it is not known whether there are a finite
number of families of Calabi--Yau threefolds; however there are evidences in \cite{ED} that there is a finite number of examples with Picard group of rank $1$. It is conjectured that all the Calabi--Yau threefold should be obtained by geometric transitions starting from Calabi--Yau threefolds with Picard group of
rank $1$ (or primitive Calabi--Yau threefolds) see \cite{Re}, \cite{G}, and \cite{Ro}. Recall also that different constructions of families of Calabi--Yau threefolds where considered in \cite{BK}, \cite{CD}, \cite{CS}, \cite{D}, \cite{Lee1}, \cite{T} see the table in Appendix \ref{A1}.

Let us describe our method; the first step is to construct Calabi-Yau threefolds with Picard
group of rank $2$ containing a del Pezzo surface $D'$. Such a
Calabi--Yau threefold is obtained as a small resolution of a nodal
Calabi--Yau threefold $X'$ containing the projection $\tilde{D}$ of
a del Pezzo surface in its anticanonical embedding. The del Pezzo
surface $D'$ can then be contracted and the resulting threefold
smoothed (by \cite{G}). We denote the Calabi--Yau threefolds
obtained by $\mathcal{Y}_t$. In comparison with the construction
using del Pezzo surfaces in their anticanonical embedding, in this
context new technical problems appear. In particular, we cannot use
Theorem 2.1 from \cite{K} to find a nodal Calabi--Yau threefold
containing the del Pezzo surface. Indeed, the base locus of hypersurfaces of
minimal degree from the ideal of the projected surface contains not
only the surface $\tilde{D}$ but also its maximal multisecant lines
(see Lemma \ref{es}). Instead, we prove Lemma \ref{lem
nodal} that can be useful in a more general context (cf.~\cite{DH}).
Throughout the paper we deal with the ideal of the projected del
Pezzo surfaces using \cite{AK,KP,P}. It turns out that the generators of the
ideals of the del Pezzo surfaces considered are closely related to the number of multisecant lines to those surfaces (analogous
relations for projections of rational normal curves were observed in
\cite{P}; see also \cite{T}) and can be studied using the geometry
of the natural nodal Calabi--Yau threefold $X'$. The number of
multisecants is computed using results of Le Barz \cite{LB,LB2}.

From the results of \cite{KK} we compute the Hodge numbers of
$\mathcal{Y}_t$ (a generic element of $\mathcal{Y}$). We also find
other important invariants of the threefolds obtained: the degree
of the second Chern class, i.e.~$c_2\cdot H$, and the degree of
the generator of the Picard group, i.e.~$H^3$, where $H$ is the
generator of the Picard group of $\mathcal{Y}_t$. The results are
presented in Table \ref{table1}; in each case we have
$h^{1,1}(\mathcal{Y}_t)=1$. We present a complete list of examples that can be obtained with this method i.e.~by projecting the del Pezzo surface into
a projective space (it is interesting to study such embedding into Grassmanians an weighted projective spaces).

\begin{center}
            %\begin{table}
            \renewcommand*{\arraystretch}{0.7}
\begin{table}[h]\label{table1}
\caption{}
\begin{tabular}{|c|c|c|c|c|c|c|}

\hline

          No & $\deg D'$&$X'$&$\sing X'$&$\chi(\mathcal{Y}_t)$&$H^3$&   $h^0(H) $                   \\ \hline
                                                                         \hline

           1& $6$&$X'_{2,4} $                           &$42$ ODP&$-96$   &$14$&$7$                              \\ \hline
2&$6$&$X'_{2,4} $ &$42$ ODP&$-98$ &$14$&$7$ \\ \hline
  3& $6$&$X'_{3,3}$                           &$36$ ODP&$-76$  &$15$&$7$                          \\ \hline
         4&   $6$&$X'_{3,3} $                           &$36$ ODP&$-78$  &$15$&$7$                          \\ \hline
5&$7$&$X'_{3,3} $                               &$44$ ODP&$-60$ &
$16$&$7$                                 \\ \hline
 6&   $8$&$X'_{3,3} $                           &$52$ ODP&$-44$  &$17$&$7$                          \\ \hline

  7&$7$&$X'_{2{,}2{,}3} $ &$37$ ODP&$-74$ &$19$&$8$                             \\ \hline

8&$8$&$X'_{2{,}2{,}3}$&$44 $ODP&$-60$&$20$&  $8$ \\ \hline

 9&           $8$&$X'_{2{,}2,2,2} $                               &$42$ ODP&$-50$  & $24$&$9$ \\ \hline
  10&          $8$&$X'_{2{,}2{,}2,2} $                               &$36$ ODP& --  & --&    --                              \\ \hline

  \end{tabular}
 \end{table}
           %\end{table}
            \end{center}
The Calabi--Yau threefolds labeled 7, 8, 9 are new. Note that Nos.~2, 4, 5, and 6 have the same numerical invariants as the examples
constructed by Tonoli \cite{T}. Nr.~2, and Nr.~4 are discussed in Example \ref{24}.  Nos.~5 and 6 are discussed in Proposition \ref{k0}, Nos.~7 and 8 are studied in Theorem \ref{kl}, and No.~9 in Theorem \ref{asd}. Note that No.~1 and No.~3  do not have smoothing in $\mathbb{P}^6$ and
the generator of their Picard group is not very ample.

The last construction from Table \ref{table1} gives an example of
a type II primitive contraction whose exceptional divisor is the
blow-up of the projective plane at a point and completes the
classification of primitive contractions of type II (see
\cite[Prob.~2.1]{K}). In this case the image of the contraction is
not smoothable.

It is natural to apply our construction in more general situations.
First, when the complete intersection $X$ containing the del Pezzo surface $D$ have only terminal singularities (not necessarily nodal).
Then we can perform after Kawamata the symbolic blow-up of $X$ by $-D$ obtaining a small resolution $Y\to X$ such that the strict transform of $D$
is isomorphic to $D$ and is a Cartier divisor on $Y$. In this case however we do not know how to find the Picard group of $Y$; it is then probably greater then $2$. Next we can apply the construction in order to construct higher dimensional manifolds. For example to construct
$4$-dimensional examples we have to find a complete intersection $X$ containing a given Fano threefold $F$. The difficulty here is that the singularities of $X$ are no more isolated and it is delicate to perform a flop of the blow-up of $F\subset X$. Note here that the cone over $F$ have a smoothing
when $F$ is a hyperplane section of a Fano $4$-fold of index $2$; those are good testing examples.

The methods developed in this paper permit us also to avoid the
computer calculations of \cite{K}. The new ingredient here is the
use of the results about the restriction of syzygies from
\cite{EGHP}, and the theory of linkages (see \cite{PS,MN}), to
compute the dimension of the linear system giving the contraction
and the degree of the image. The use of the computer algebra system
Singular \cite{GPS} is, however, needed to compute Gr\"{o}bner bases
in the larger context studied in this paper (useful Singular scrips are available at the end of the paper). To prove several
statements, we find using Singular an example where the statements
hold, and then use semi-continuity arguments.

\section*{Acknowledgements} The author would like to thank Micha\l\, Kapustka
for many discussions. He would also like to thank P.~Le Barz,
G.~Bini, S.~Cynk, B.v.~Geemen, Ch.~Okonek, V.~Przyjalkowski, J.~Wi\'{s}niewski
for answering questions and motivations. He thank the referee for helpful remarks. Part of this work was done
during his stay at the University of Z\"{u}rich.
\section{Projections of del Pezzo surfaces}
  We are interested in minimal resolutions of the ideals of generic linear
projections of del Pezzo surfaces. Denote by $\tilde{D}\subset
\mathbb{P}^{N-1}$ the projection of $D\subset \mathbb{P}^N$ from a
generic point of $\mathbb{P}^N$ (in particular
$\tilde{\tilde{D}}\subset \mathbb{P}^{N-2}$ is a projection from a
generic line). Denote by $D_i\subset \mathbb{P}^i$ the del Pezzo
surface of degree $i$ for $i\leq 7$, by $D_8$ the double Veronese
embedding of a quadric and by $\mathbb{F}_1$ the blow-up of
$\mathbb{P}^2$ in one point (embedded by the anticanonical system).
\begin{prop}\label{po}\begin{enumerate}\item The ideal of $\tilde{\mathbb{F}}_1$ in the homogeneous coordinate
ring of $\mathbb{P}^7$ is generated by $11$ quadrics and one cubic.
\item The ideal of $\tilde{D_8}\subset \mathbb{P}^7$ is generated by
$11$ quadrics.
\item The ideal of $\tilde{D_7}\subset \mathbb{P}^6$ is generated
by $6$ quadrics and $3$ cubic.
\item The ideal of
$\tilde{D_6}\subset \mathbb{P}^5$ is generated by $2$ quadrics and
$7$ cubics.
\end{enumerate}
\end{prop}
\begin{proof} It is well known (see \cite{Hoa}) that the minimal free
resolutions of the del Pezzo surfaces $D_6\subset \mathbb{P}^6$,
$D_7\subset \mathbb{P}^7$, and $D_8\subset \mathbb{P}^8$ are the
following:
$$ \mathcal{O}^9(-2)\leftarrow \mathcal{O}^{16}(-3)\leftarrow \mathcal{O}^9(-4)\leftarrow \mathcal{O}(-6)\leftarrow 0 $$
$$\mathcal{O}^{14}(-2)\leftarrow \mathcal{O}^{35}(-3)\leftarrow \mathcal{O}^{35}(-4)\leftarrow \mathcal{O}^{14}(-5)
\leftarrow \mathcal{O}(-7)  \leftarrow  0 $$
$$\mathcal{O}^{20}(-2)\leftarrow
\mathcal{O}^{64}(-3)\leftarrow \mathcal{O}^{90}(-4)\leftarrow
\mathcal{O}^{64}(-5)\leftarrow \mathcal{O}^{20}(-6)\hookleftarrow
\mathcal{O}(-8)$$
 respectively. Moreover, $\mathbb{F}_1$ has the free resolution of the same shape as $D_8$. In particular these surfaces satisfy property $N_{2,3}$ thus
it follows from \cite[Thm.~3.1 (b)]{AK} that the projected surfaces have
ideals generated by cubics. Now the number of quadrics in the ideals is
computed using \cite[Prop.~4.1 (a)]{AK} and found to be equal to $h^0(\mathcal{I}_{\tilde{D_d}(2)})=h^0(\mathcal{I}_{D_d}(2))-d$. So the problem is to find the number of
cubics or equivalently to find the number of all generators (this is the $0$-th Betti number $b_0(\mathcal{I}_{\tilde{D}}))$.

The idea is to compute this number for a concrete example
and use the Zariski semicontinuity of the graded Betti numbers on
the open set where the Hilbert function is maximal (proven in
\cite[Prop.~2.15]{BG}). First, from \cite[Rem.~4.2]{AK} we find that $h^0(\mathcal{I}_{\tilde{D_d}}(k))=h^0(\mathcal{I}_{D_d}(k))-$ ${d+k-1}\choose {d}$.
It follows that the Hilbert $H(n)=h^0(\mathcal{O}_{\mathbb{P}^{d-1}}(n))-h^0(\mathcal{I}_{\tilde{D}}(n))$ does not depend of the center of projection and that
the Hilbert function is constant for isomorphic projections.
We shall prove this proposition in the easiest and the harder case.

Let us consider the
example above in the case of $D_8\subset \mathbb{P}^8$ with
coordinates ($x,y,z,t,u,v,w,s,m$).
 Such a surface is given by the
$2\times 2$ minors of a symmetric $4\times 4$ matrix whose entries
are linear forms of the nine coordinates. Let us project from the
point $(1,0,\dotsc,0)$ the surface given by the matrix
$$\left( \begin{array}{c}
\begin{array}{cccc}
x&y+x&z&t\\
 y+x&u&v&w\\
 z&v&s&m-x\\
 t&w&m-x&s
\end{array}
\end{array}\right).$$
We can calculate by hand (or with Singular) by eliminating the
variable $x$ that the projected surface needs only quadric
generators. Since by the discussion before the number of generators can be only smaller for special values we are done (because the number of quadrics is constant and equal to $h^0(\mathcal{I}_{\tilde{D_8}}(2))$).

Let us consider now the most difficult case $3$. To argue as before we need to find a
lower bound of the number of cubics. Consider the system $|2H-E_1-E_2|$ on $D_7$, where $E_1$, $E_2$
are exceptional divisors and $H$ the pull-back of the hyperplane
section from $\mathbb{P}^2$, is a system of rational normal curves
of degree $4$. This system defines a $3$-dimensional family of
four dimensional projective spaces in$ \mathbb{P}^7$ spanned by the rational curves. Since the trisecant
planes to $D_7$ cover all $\mathbb{P}^7$ we can find a rational
quartic curve passing through $3$ generic points of $D_7$ thus the above $\mathbb{P}^4$'s covers the ambient space $\mathbb{P}^7$. So we
can find one considered $\mathbb{P}^4$ passing through the center of
projection. We are interested what is happening after projecting this $\mathbb{P}^4$.
Let us remind that the projected rational
normal curve of degree $4$ in $\mathbb{P}^3$ is a divisor of
bi-degree $(1, 3)$ in a smooth quadric $\mathbb{P}^1 \times
\mathbb{P}^1$ thus needs $3$ cubic generators. This gives the required lower bound.\end{proof}
By computing the Chern classes of the conormal bundle of $\tilde{D_8}\subset \mathbb{P}^7$ and using \cite[Thm.~2.1]{K} we obtain the following:
\begin{cor} The complete intersection of four quadrics containing $\tilde{D_8}$ is a nodal Calabi--Yau threefold $X$ with $36$ ODP.
\end{cor}
Since the syzygies between the quadric generator are not all linear we do not obtain directly that the Picard group of the Calabi--Yau threefold obtained by blowing up $X$ along $\tilde{D_8}$ is $2$. To prove this we need to analyze the exceptional set of the map defined by the quadric containing $\tilde{D_8}$ this will be treated in Theorem \ref{asd}.
\section{degree $15$ Calabi--Yau threefolds}\label{poi}
The methods and remarks from this section can also be applied to
Calabi--Yau threefolds considered in \cite{K}.
 Let $D_6\subset \mathbb{P}^6$ be an anticanonically embedded del Pezzo surface of
degree $6$. Denote by $\tilde{D}$ the projection of $D_6$ into
$\mathbb{P}^5$ from a generic point $Q$ in $\mathbb{P}^6$. It
follows from Proposition \ref{po} that the ideal of $\tilde{D}\subset
\mathbb{P}^5$ is generated by cubics. From \cite[Thm.~2.1]{K}, we
infer that the generic complete intersection $X'\subset
\mathbb{P}^5$ of two cubics containing $\tilde{D}$ is a nodal
Calabi--Yau threefold. Using Chern classes we compute the $36$ nodes
on $X'$. Denote by $S'$ the surface linked via a general cubic to
$\tilde{D}$ on $X'$, i.e. $S'\in|3H-\tilde{D}|$ where $H$ is the
hyperplane section of $X'\subset \mathbb{P}^5$.
\begin{lemm}\label{lem1} The surface $S'\subset \mathbb{P}^5$ is smooth and is contained in a
quartic that does not contain $\tilde{D}$.
\end{lemm}
\begin{proof} From \cite[Prop.~4.1]{PS} we deduce that $S'$ is
smooth. Next we prove that there is a quartic in the ideal of
$S'\subset \mathbb{P}^5$ that is not generated by the three cubics
defining $S'\cup \tilde{D}$. From the following standard liaison
exact sequence $$0\rightarrow \omega_{\tilde{D}}(1) \rightarrow
\mathcal{O}_{S'\cup \tilde{D}}(4)\rightarrow
\mathcal{O}_{S'}(4)\rightarrow 0$$ we infer
$h^0(\mathcal{I}_{S'}(4))>h^0(\mathcal{I}_{S'\cup \tilde{D}}(4))$
since $\omega_{\tilde{D}}=\mathcal{O}_{\tilde{D}}(-1)$.\end{proof}
 Let $G'$ be the smooth surface linked to $S'$ via a
general quartic on $X'\subset \mathbb{P}^5$. Denote by $X$ the
Calabi--Yau threefold obtained by flopping the exceptional curves of
the blowing-up of $X'$ along $\tilde{D}$. Let $D'$ and $G$ be the
strict transforms on $X$ of $\tilde{D}$ and $G'$ respectively.

\begin{prop}\label{prop deg15} The image of $X$ under the morphism $\varphi_{|G|}$ is a
threefold $Y\subset\mathbb{P}^6$ of degree $15$ with one singular
point $P$. Moreover, $X'\subset \mathbb{P}^5$ is the projection of
$Y$ from $P$.
\end{prop}
\begin{proof} First, $G\in|H^*+D'|$ where $H^*$ is the pull-back of
$H$ on $X$. So $|G|$ is very ample outside $D'$. From Lemma
\ref{lem1} we infer that the effective divisors $D'\subset X$ and
$G\subset X$ do not have common components. Since $G|_{D'}$ is
trivial we obtain $D'\cap G=\emptyset$. So $|G|$ is base-point-free
and contracts $D'$ to a point.

To see that $h^0(\mathcal{O}_X(G))=7$ we need to prove that
$G'\subset \mathbb{P}^5$ is linearly normal
(cf.~\cite[Lem.~2.1]{K}). This follows from the fact that $G$ and
$\tilde{D}$ are doubly linked so we have
$H^1(\mathcal{I}_{\tilde{D}}(k))=H^1(\mathcal{I}_{G'}(k+1))$ for
$k\in\mathbb{Z}$ (cf.~\cite[Cor.~5.11]{MN}).

Finally, the projection of $Y$ from $P$ can be seen as the image of
$X$ under the linear subsystem of $|D'+H^*|$ of dimension $6$ with
$D'$ being a fixed component, thus under $|H^*|$ .
\end{proof}
\begin{rem}\label{rem1} The threefold $Y$ is not normal at $P$. We need to take a multiple of $G$ to obtain a primitive
contraction (cf. \cite[Lem.~2.5]{KK}). However, it is possible that
that in some cases $Y$ can be smoothed by Calabi--Yau threefolds in
$\mathbb{P}^6$. Note that we know that the germ of the cone over a
projected del Pezzo surface of degree $6$ can be smoothed by taking
hyperplane sections of the cone over the projection of
$\mathbb{P}^1\times \mathbb{P}^1\times \mathbb{P}^1$.
\end{rem}
\begin{thm}\label{we} The morphism $\varphi_{|2G|}$ gives a primitive
contraction with image being a singular Calabi--Yau threefold that
is a degeneration two family of Calabi--Yau threefolds with
$h^{1,2}=39$ and $h^{1,2}=40$ of degree $15$. Moreover, the Picard
groups of the threefolds obtained are isomorphic to $\mathbb{Z}$.
\end{thm}
\begin{proof} If we prove that $\varphi_{|2G|}$ is a primitive contraction, then from \cite{G} the resulting singular Calabi--Yau threefolds
can be smoothed. The problem is to show the normality of the
image. Since from \cite[Thm.~3.1(a)]{AK} we have $h^1(\mathcal{I}_{\tilde{D}}(2))=h^1(\mathcal{I}_{D}(2))=0$ the restriction
$\mathcal{O}_{\mathbb{P}^7}(2H)\twoheadrightarrow\mathcal{O}_{\tilde{D}}(2H)$
is surjective. We can conclude as in the proof of
\cite[Thm.~2.3]{K}.

Let us show that the rank $\rho(X)$ of the Picard group of $X$ is
$2$ and that this group is torsion free. First, from
\cite[Thm.~3.1]{AK}, we see that $\tilde{D}$ is $3$-regular so the syzygies between the cubics
containing $\tilde{D}$ are linear. Let $C\supset \tilde{D} $ be a
general smooth cubic in $\mathbb{P}^5$.
 Arguing as in the proof of \cite[Thm.~2.2]{K} we infer that it
is enough to prove that the morphism $\pi$ obtained from the system
of cubics on $C$ containing $\tilde{D}$ does not contract any
divisor to a curve. From \cite[Prop.~3.1]{AR} the two-dimensional
fibers of $\pi$ are planes cutting $\tilde{D}$ along cubic curves.
Such planes are contained in the sum $S_3$ of trisecant lines.
Observe that these trisecant lines are images of trisecant planes to
$D_6$ passing through the center of the projection $Q$ (from
\cite[Thm.~1.1]{EGHP} there are no trisecant lines to $D_6$ since
this surface is cut out by quadrics). It follows that the dimension
of $S_3$ is $3$. Since the system of cubics containing $\tilde{D}$
is base-point-free on $\mathbb{P}^5-\tilde{D}$ we see that there is
at most a one-dimensional family of trisecant lines to $\tilde{D}$,
thus there is no divisor on $C$ that contracts to a curve.

Denote by $Y\subset \mathbb{P}^N$ the image $\varphi_{|2G|}(X)$ and
by $\mathcal{Y}_t$ ($t\in \mathbb{C}$ in the neighborhood of $0$) a
generic element of the smoothing family $\mathcal{Y}$ of $Y$. First
from the proof of \cite[Prop.~3.1]{KK} we have
$H^2(Y,\mathbb{Z})\simeq H^2(\mathcal{Y}_t,\mathbb{Z})$. We claim
that these cohomology groups are torsion free. Indeed, if
$\mathcal{L}$ is a torsion sheaf on $\varphi (X)$ then
$\varphi^*(\mathcal{L})$ is torsion on $X$, moreover it is non-zero
form the projection formula ($\varphi(X)$ is normal). It remains to
recall that $H^2(X,\mathbb{Z})$ is torsion free. Next, the image $T$
of $G$ on $Y$ is an ample divisor ($2T$ is very ample) such that
$T^3=15$ and $h^0(\mathcal{O}(T))=7$ (by Proposition \ref{prop
deg15}). From the discussion in \cite[\S 3]{Wi} we obtain, for $t$
sufficiently small, a flat family of ample (and base-point-free)
divisors $T_t\subset \mathcal{Y}_t$ such that $T_0=T$. It follows
that $T_t^3=15$.

Let us compute $c_2 \cdot T_t$ where $T_t$ is the generator of
$H^2(\mathcal{Y}_t,\mathbb{Z})$. From \cite[Lem.~2.2]{K} we can
embed $\mathcal{Y}_t$ into $\mathbb{P}^N$. This embedding is
clearly given by the complete linear system $2T_t$, since
$h^{1,1}(\mathcal{Y}_t)=1$ and the embedding $\mathcal{Y}_t\subset
\mathbb{P}^N$ is linearly normal
($h^1(\mathcal{I}_{Y|\mathbb{P}^N}(1))=0$). Since $c_2\cdot
2T=108$, we infer $c_2\cdot T_t=54$; this ends the proof.
\end{proof}
\begin{rem}\label{e1} The above families have the same invariants as the
following Calabi--Yau threefolds: \begin{enumerate} \item of
degree $15$ constructed by Tonoli (see \cite{T}); \item of
degree $15$ constructed by Lee in \cite{Lee1} as double cover of a
singular Fano threefold. \end{enumerate} It would be interesting
to know whether the families obtained are exactly the above ones.
\end{rem}
\begin{rem} The Calabi--Yau threefold $X$ is birational to another
interesting Calabi--Yau threefold. Consider the blow-up $Z$ of $X'$
along $\tilde{D}$. Then $Z$ is a Calabi--Yau threefold with Picard
group of rank $2$ thus admits a second primitive contraction or a
fibration. This morphism is given by a multiple of the pull-back
of $T\in|kH-\tilde{D}|$ on $Z$, for some $k\in\mathbb{Z}$, and can
be studied by liaison methods as above. We obtain in this way many
examples of Calabi--Yau threefolds and geometric transitions. This
will be discussed elsewhere.
\end{rem}
\begin{ex}\label{24} In an analogous way, we can embed $\tilde{D}$ into a
Calabi--Yau threefold which is a complete intersection of a quadric
and a quartic. The problem is to prove that in this case we obtain a
nodal Calabi--Yau threefold. Then it must have $42$ nodes, thus the
resulting Calabi--Yau threefolds of degree $14$ have Euler
characteristic $-96$ and $-98$ (probably the degree $14$ examples
from \cite{T} and \cite{Lee1}). In order to use \cite[Thm.~2.1]{K}
we have to prove the following.
\begin{lemm} A generic quadric from the ideal of $\tilde{D}\subset
\mathbb{P}^5$ is smooth.
\end{lemm}
\begin{proof} It is enough to prove that the generic singular
quadric from the ideal of $D_6\subset\mathbb{P}^6$ is a cone over a
smooth quadric. Recall that the ideal of a del Pezzo surface of
degree $6$ is given by $2\times 2$ minors of a $3\times 3$ matrix
whose entries are linear forms. We find explicitly a cone over a
smooth quadric in this ideal.
\end{proof}
\begin{rem} We compute using Singular that the intersection of two
generic quadrics from $\mathcal{I}_{\tilde{D}}(2)$ is a nodal Fano
threefold with $6$ nodes. It would be interesting to try to apply our
construction to this threefold.
\end{rem}
\end{ex}
\section{degree $24$ Calabi--Yau threefolds}
Let us consider the projections of del Pezzo surfaces of degree $8$.
There are two of them, the Hirzebruch surface $\mathbb{F}_1\subset
\mathbb{P}^8$ and $D_8\subset \mathbb{P}^8$. Denote by $\tilde{D_8}$
and $\tilde{\mathbb{F}_1}$ the projections of $D_8$ and $\mathbb{F}_1$ into
$\mathbb{P}^7$ from a generic point $Q$ in $\mathbb{P}^8$. We shall
embed $\tilde{D_8}$ and $\tilde{\mathbb{F}_1}$ into nodal complete intersections
of four quadrics in $\mathbb{P}^7$. In these cases however more
technical problems arise. We need the following lemma.
\begin{lemm}\label{lem nodal} (cf.~\cite{DH}) Let $X\subset \mathbb{P}^n$ be a reduced two-dimensional sub-scheme whose
ideal is generated by hypersurfaces of degree $d$. Assume that the
 scheme $X$ has no embedded components, and has exactly
one two-dimensional component $X_c$ such that the other components
have smaller dimension and intersect $X_c$ transversally. Then the
generic complete intersection of $n-2$ hypersurfaces of degree $d$
from the ideal of $X$ is a nodal variety with nodes lying on $X$.
\end{lemm}
\begin{proof} This proof is a generalization of \cite[Thm.~2.1]{K}.
We need only prove that if $X_e$ is a component of codimension
$>n-2$ then the intersection of $n-2$ hypersurfaces from
$H^0(\mathcal{I}_X(d))$ is non-singular along $X_e$. Let us
consider the following diagram:
\[\begin{array}{ccc}\mathbb{P}^n &\dashrightarrow & \mathbb{P}^{N}\\
\uparrow& \nearrow&\\
 \overline{\mathbb{P}^n}&&
\end{array}
\]
where the vertical map is the blow-up of $\mathcal{I}_{X}$, denoted
by $\pi$. The variety $\overline{\mathbb{P}^n}$ can be seen as
the closure of the graph of the morphism given by
$(q_0:\dots:q_N)$ (here $q_0,\dots,q_N$ are the degree $d$
generators of the ideal of $X$). The horizontal map $\beta$ is
given by the linear system $H^0(\mathcal{O}(d)\otimes
\mathcal{I}_{X})$. The remaining map is the projection $p\colon
\mathbb{P}^n\times\mathbb{P}^N\rightarrow \mathbb{P}^N$.

The complete intersections $C$ containing $X$ are the
pre-images of linear spaces $L_C$ in $\mathbb{P}^N$. A singularity
appears on a normal $C$ (in particular generic, see \cite{DH}) at
$q\in X_e-X_c$ iff $L_C$ intersects the linear space $p(\pi^{-1}(q))$
non-transversally. For dimensional reasons a generic $L_C$ does
not contain any such linear space.

 Let us consider the points $q\in X_e\cap X_c$. Locally analytically the blow-up of $\mathbb{C}^{n}$ along the union $S$ of a line and a plane passing through $0$ has
 $\mathbb{P}^{n-2}$ as exceptional locus over $0$ with two distinguished coordinates corresponding to
quadrics from the ideal of $S$. Moreover, an analytic germ at $0$
containing $S$ is smooth if its strict transform on the
exceptional divisor is smooth and meets transversally the
codimension $2$ linear space determined by the two distinguished
coordinates. We conclude that the exceptional divisor
$\pi^{-1}(q)\subset \mathbb{P}^N$ is isomorphic to
$\mathbb{P}^{n-2}$ with two distinguished coordinates. From the
Bertini theorem a complete intersection corresponding to $L_C$ is
smooth at $q$ for a generic choice of $L_C$.\end{proof}
\begin{prop}\label{lem24}  The intersection of the $11$
quadrics from the homogeneous ideal $\mathcal{I}_{\tilde{\mathbb{F}_1}}\subset
\mathbb{P}^7$ defines scheme-theoretically the union of $\tilde{\mathbb{F}_1}$
and the unique trisecant line to $\tilde{\mathbb{F}_1}$ that is transversal to
$\tilde{\mathbb{F}_1}$.
\end{prop}
\begin{proof} It follows from \cite[Example 5.16]{CR} that there
is exactly one trisecant plane to $\mathbb{F}_1$ (it is transversal
to $\tilde{\mathbb{F}_1}$) passing through a generic point in $\mathbb{P}^8$,
so there is a unique trisecant line $t$ to $\tilde{\mathbb{F}_1}$.

Each cubic containing $t\cup \tilde{\mathbb{F}_1}$ is of the form
$$a_1q_1+\dotsc+a_{11}q_{11}+b\cdot c$$ where $q_1,\dotsc,q_{11}$
are the quadric generators of $\tilde{\mathbb{F}_1}$, $c$ the cubic generator,
$a_1,\dotsc,a_{11}$ are linear forms, and $b$ is a constant. It
follows that $b=0$ since $c|_t\neq 0$. It is enough to prove that
the ideal defining $t\cup \tilde{\mathbb{F}_1}$ is generated by cubics. We need
the following generalization of \cite[Cor.~3.4]{HK}.
\begin{lemm}\label{HK} Let $X\subset \mathbb{P}^n$ be a non-degenerate
projective variety satisfying property $N_{3,p}$ and $q$ be a
smooth point of $X$. Suppose that $\mathcal{I}_X$ is generated by
quadrics. Consider the inner projection $\pi_q\colon X\rightarrow
Y\subset\mathbb{P}^{n-1}$. Then the projected variety $Y$ satisfies
property $N_{3,p-1}$.
\end{lemm}
\begin{proof} The proof is analogous to the proof of \cite[Thm.~3.1]{HK}.
Indeed, we have an exact sequence
$$0\to I_Y\to I_X\to I_X/I_Y\to o$$
 of $S:=k[x_0,\dots,x_{n-1}]$-modules where $I_X$ and $I_Y$ are the ideals of $X$ an $Y$.

 Suppose that we proved that $Tor^S_{l-2}(I_Y,k)_{l+2}=0$ for some $l\leq p$ we will show that $Tor^S_{l-1}(I_Y,k)_{l+3}=0$. From the corresponding long exact sequence of Tor's we obtain the sequence:
 $$ Tor^S_l(I_X,k)_{l+3}\xrightarrow{a} Tor^S_l(I_X/I_Y,k)_{l+3}\to Tor^S_{l-1}(I_Y,k)_{l+3}$$
$$ \to Tor^S_{l-1}(I_X,k)_{l+3}\xrightarrow{b} Tor^S_{l-1}(I_X/I_Y,k)_{l+3}\to 0.$$
It is enough to show that $b$ is an injective and $a$ is surjective.
Following \cite{HK} we have a diagram
\[\begin{array}{ccc} Tor^S_l(I_X,k)_{l+3} & \xrightarrow{a} & Tor^S_l(I_X/I_Y,k)_{l+3}\\
\downarrow &&\downarrow \\

Tor^S_l(I_X,k)_{l+N}  & \xrightarrow{c}  & Tor^S_l(I_X/I_Y,k)_{l+N}
\end{array}
\]
where $N$ is such that $Tor^S_i(I_Y,k)_{i+N}=0$ for all $i\in\mathbb{Z}$; thus such that $c$ is surjective.
Since $\mathcal{I}_X$ is generated by
quadrics, the right vertical map is surjective from \cite[Prop.~2.5 (b)]{HK}. The left vertical map is surjective from the elimination mapping cone sequence from \cite[Thm.2.1(b)]{HK}. The injectivity of $a$ is proven similarly.
\end{proof}
To use Lemma \ref{HK} we need to show that the ideal of
$\mathbb{F}_1\cup L\subset \mathbb{P}^8$, where $L$ is a general
trisecant plane, is generated by quadrics. Let $\mathbb{P}^8\supset
H$ be a $3$-dimensional linear space containing $L$. From
\cite[Thm.~1.1]{EGHP} it follows that $H$ cuts $\mathbb{F}_1$ in at
most four points. Since $\mathbb{F}_1$ satisfies property $N_{2,5}$
we deduce from \cite[Thm.~1.2]{EGHP} that each reducible quadric $W$
containing $L$ and the schematic intersection of $H$ and
$\mathbb{F}_1$, is the restriction of one from the $17$-dimensional
family $\Gamma$ of quadrics containing $\mathbb{F}_1$. It follows
that the scheme determined by the quadrics from $\Gamma$ cannot have
other components then $\mathbb{F}_1$ and $L$ (there are no embedded
components with support contained in $L\cap \mathbb{F}_1$ since $H$
can be chosen generically thus this component would be contained in
$L$).\end{proof} From Lemma \ref{lem nodal} the generic intersection
$X_1'$ (resp.~$X_2'$) of four quadrics containing $\tilde{D_8}$
(resp.~$\tilde{\mathbb{F}_1}$) is a nodal Calabi--Yau threefold with $42$
(resp.~$36$) nodes on $\tilde{D_8}$ (resp.~$\tilde{\mathbb{F}_1}$). We denote by
$X_1$ (resp.~$X_2$) the Calabi--Yau threefold obtained by the
flopping of the exceptional curves of the blow-up $X_1''\rightarrow
X'_1$ (resp. $X_2''\rightarrow X'_2$), and by $D$' (resp.~$F'$) the
strict transforms of $\tilde{D_8}$ (resp.~$\tilde{\mathbb{F}_1}$).
\begin{thm}\label{asd} The Calabi--Yau threefolds $X_1$ and $X_2$ have Picard
group of rank $2$. There exist primitive contractions
$X_1\rightarrow Y_1$ and $X_2\rightarrow Y_2$ with exceptional
divisors $D'$ and $F'$ respectively. Moreover, the threefold $Y_1$
can be smoothed by a family of Calabi--Yau threefolds of degree $24$
with Picard groups of rank $1$ and $h^{1,2}=26$.
\end{thm}
\begin{proof} Since the syzygies between the quadric generators of
$\tilde{D_8}$ and $\tilde{\mathbb{F}_1}$ are not generated by linear forms, we
cannot compute $\rho(X_1)$ and $\rho(X_2)$ as before. Let us
concentrate on $X_1$.

Consider the cone $C$ over $D_8\subset \mathbb{P}^8$ with vertex $Q$
being the center of the projection. Let $\alpha \colon \mathbb{P}^8-
C\rightarrow \mathbb{P}^{10}$ be the morphism given by the system of
quadrics containing the cone $C$.
\\ \emph{Claim}:
The morphism $\alpha$ is an embedding outside the subset
$Join(D_8,C) $ (i.e.~the sum of lines joining points on $D_8$ with
points on $C$) of codimension $\geq 2$.

We show that the image under the morphism $\alpha$ of a line
$l\subset \mathbb{P}^8$ that is not contained in $Join(D_8,C)\subset
\mathbb{P}^8$ is a line or a conic.
 Let $\beta\colon \mathbb{P}^8 \dashrightarrow \mathbb{P}^{19}$
be given by the system $H^0(\mathcal{O}(2)\otimes
\mathcal{I}_{D_8})$. From \cite[Prop.~3.1]{AR}, it is an embedding
off $Sec(D_8)$. By Proposition \ref{lem24} the image $\beta(C)$ is
contained in a $9$-dimensional linear space $L$ such that $L$ cuts
exactly $\beta(C)$ out of the closure of $\beta(\mathbb{P}^8)$. The
morphism $\alpha$ can be seen as the composition of $\beta$ with the
projection from $L$.
 From \cite[Thm.~1.2]{EGHP} the image $\beta(l)$ is
 either a plane conic disjoint from $\beta(C)$
so also disjoint from $L$, or a line disjoint from $L$ (if $l$
intersect $D_8$). The claim follows since any $10$-dimensional
linear space containing $L$ meets $\beta(l)$ in zero, one or, if
$\beta(l)$ spans a plane that intersects $L$, in two points.

From \cite[Thm.~6]{RS} we obtain $\rho(X_1)=\rho(X_2)=2$; the other
claim follows as before.\end{proof}
\begin{rem} The example  with
$\mathbb{F}_1$ as exceptional locus completes the classification in
\cite[Thm.~2.5]{K} of exceptional loci of primitive contractions of
type II. The image of such contraction is not smoothable.
\end{rem}
\begin{rem} Let us project $D_8$ from a point in
$Sec_3(D_8)-Sec_2(D_8)$. The resulting surface $\tilde{D_8}$ has then
one cubic and $11$ quadric generators. Since $Sec_3(D_8)$ is
degenerate, the intersection of the $11$ quadrics is the union of
$\tilde{D_8}$ and a plane intersecting $\tilde{D_8}$ along a cubic. We
cannot perform the previous construction in this case since Lemma
\ref{lem nodal} does not work.
\end{rem}
\section{higher codimension projections}\label{OE}
In this section we consider all the remaining cases when our method works and produces Calabi--Yau threefolds with $h^{1,1}=1$.
In order to obtain smoothable Calabi--Yau threefolds we treat the del Pezzo surfaces $D_7$ and $D_8$.
By analyzing the Betti tables of projections of del Pezzo surfaces we see that we have four remaining possibilities in order that the generic projection of a del Pezzo surface is contained in a nodal Calabi--Yau complete intersection. The possibilities are $\tilde{\tilde{D_7}}\subset\mathbb{P}^5$, $L_P\subset \mathbb{P}^5$ (the projection of $D_8\subset \mathbb{P}^8$ from the plane $P\subset\mathbb{P}^8$), $\tilde{D_7}\subset\mathbb{P}^6$, $\tilde{\tilde{D_8}}\subset\mathbb{P}^6$.
The latter two cases leads to constructions of new families of Calabi--Yau threefolds.
\begin{thm}\label{kl} There exist a Calabi--Yau threefold with Picard group of rank $1$ of degree $19$ (resp.~$20$) with $h^{1,2}=23$ (resp.~$h^{1,2}=31$).
\end{thm}
\begin{proof}
 The plan of the proof is the following: we find explicitly the centers of projections such that the surfaces $\tilde{\tilde{D_8}}\subset\mathbb{P}^6$ and $\tilde{D_7}\subset \mathbb{P}^6$ are embedded into nodal Calabi--Yau threefolds in
$\mathbb{P}^6$ which are complete intersections $X_1$ (resp.~$X_2$) of two quadrics and
a cubic with $37$ (resp.~$44$) nodes. Next analogically as before we show that the Picard group of small resolutions of $X_1$ (resp.~$X_2$) has rank $2$. The Calabi--Yau threefolds we are looking for are obtained after the smoothing of the primitive contractions of $X_1$ (resp.~$X_2$).

Let us consider the harder case $\tilde{\tilde{D_8}}$.
We find explicitly using Singular with a random choice of the center of projection, two quadrics containing a
projected surface intersecting each other along a smooth threefold $Y$ (it would be interesting to prove that this holds for a generic choice of the center). From the Lefshetz theorem the Picard group of the intersection is $1$.
To prove that the threefold $X_1$ has nodes it is enough to show, by same arguments as in Lemma \ref{lem nodal}, that the ideal $\mathcal{I}_{\tilde{\tilde{D_8}}|Y}$ is generated by cubics.
To compute the rank of the Picard group of $X_1$ we use \cite[Thm.~2.1]{K}, however
to make it work we need to show, that the projected surfaces are $3$-regular. Indeed, we compute that for a random choice of the center of projection the Betti table of the resolution of the ideal of $\tilde{\tilde{D_8}}$ is the following:
\begin{center}
            %\begin{table}
            \renewcommand*{\arraystretch}{0.7}
\begin{table}[h]

$\begin{tabular}{|ccccccc}
1&0&0&0&0&0&0\\
0&3&0&0&0&0&0\\
0&14&53&68&43&14&2
  \ \end{tabular}$
 \end{table}
           %\end{table}
            \end{center}
 It follows that the syzygies between the cubics are linear thus the exceptional set of the morphism given by the cubics containing $\tilde{\tilde{D_8}}$ (resp.~$\tilde{D_7}$) is well described.
In particular this linear system does not contract divisors on the intersection of our two quadrics. We can conclude as in the proof of Theorem \ref{we}.
\end{proof}
\begin{rem} From the proof below we deduce that the general projection $\tilde{\tilde{D_8}}\subset \mathbb{P}^6$ needs $14$ cubics an $3$ quadrics generators.
\end{rem}
Finally let us consider the remaining two projections. It is interesting that in those cases we obtain Calabi--Yau threefolds with the same invariants as the examples of Calabi--Yau threefolds in $\mathbb{P}^6$ constructed by Tonoli.
To prove this we have to understand the properties of the projected del Pezzo surfaces.
From \cite[Cor.~3.1]{AK} we know that the projection from a line (resp.~plane) is $4$-regular (resp. $5$-regular); we can prove more. The following proposition can be checked for a random center of projection using Singular; we aimed however to present here a formal proof.
 \begin{prop}\label{es}\begin{enumerate}\item The ideal of the surface $\tilde{\tilde{D_7}}\subset \mathbb{P}^5$ is generated by $1$ quartic and $13$
cubics and this surface has exactly one quadrisecant line.
\item The ideal defining $L_P\subset \mathbb{P}^5$ needs $7$
quartic and $7$ cubics generators, and this surface has
exactly $7$ quadrisecant lines.\end{enumerate}
Moreover,the intersection of the
cubics from the homogeneous ideal of $\tilde{\tilde{D_7}}\subset \mathbb{P}^5$
(resp.~$L_P\subset \mathbb{P}^5$) defines scheme-theoretically the
union of $\tilde{\tilde{D_7}}$ and the unique quadrisecant line (resp.~the union of
$L_P$ and the quadrisecant lines).
 \end{prop}
\begin{proof} First to compute the number $q(V)$ of quadrisecant lines to a surface $V\subset \mathbb{P}^5$ let us introduce after Le Barz the invariants
$n=deg(V)$, $d$, $\delta$, $t$. Denote by $V'\subset \mathbb{P}^4$ (resp.~$V''\subset \mathbb{P}^3$) the generic projection of $V$ from a point (resp.~a line). Then $d$ is the degree of the double curve on $V''$, $t$ the number of triple points on $V''$, and $\delta$ the number of double points on $V'$.
If $V\subset \mathbb{P}^5$ do not contain lines (like in the case of $L_P$) then by \cite{LB2} the number of quadrisecant lines is computed as follows
$$q(V)=13 {n\choose 4}-3n(n-4)(2n-3)+t(2n-27)+{\delta\choose 2}+\delta(7-2n)+{d\choose 2}-d(2n^2-29n+83).$$
Since $D_7$ contains one line with self intersection $-1$ we have to subtract $1$ in the formula above.
To compute the invariants $d$, $\delta$, $t$ we use either Kleiman's multiple points formulas or \cite[p.~59]{LB1}. In particular for the surface $L_P$ we obtain $d = 20$, $\delta = 10$ and $t = 20$.

We shall now prove that there is $7$ cubic and at most $7$ quartics generators of $L_P$ for generic $P\in U\subset G(3,9)$.
Let $D\subset \mathbb{P}^5$ be a del Pezzo surface embedded by a subsystem of the anti-canonical system.
From the Riemann--Roch theorem we deduce that $$h^0(\mathcal{O}_D(n))=\frac{1}{2}n(n+1)\deg(D)+1$$ for $n\geq 0$ this gives us the difference
$h^0(\mathcal{I}_D(n))-h^1(\mathcal{I}_D(n))$.
Let us consider the harder case $(2)$.
  We compute for a random example that $h^0(\mathcal{I}_{L_P}(2))=0$ and $h^0(\mathcal{I}_{L_P}(3))=7$.
 Using the exact sequence \begin{equation}\label{q2} 0\to \mathcal{I}_{L_P}\to \mathcal{O}_{\mathbb{P}^5}\to \mathcal{O}_{L_P}\to 0\end{equation} we find that
$h^1(\mathcal{I}_{L_P}(k))=0$ for $k\geq 3$ moreover $h^1(\mathcal{I}_{L_P}(1))=3$ and $h^1(\mathcal{I}_{L_P}(2))=4$.
By the upper semi-continuity of the function $h(P)=h^l(\mathcal{I}_{L_P}(k))$ (for each $k,l$ discussed in \cite{BG}) we deduce that these equalities holds for a generic $P\in U$.
We deduce also that $H^i(\mathcal{I}_{L_P|\mathbb{P}^5}(4-i))=0$ for $i\geq 1$ thus by the Mumford criterium $L_P$ is $4$-regular and in particular generated by quartics for generic $P$. It follows also that the Hilbert function is constant on $U$ and we can use the semi-continuity of Betti numbers from \cite{BG}.
 We find that for a special choice of $P$ the surface $L_P$ have $7$ cubic and $7$ quartic generators.
We deduce that generically $I_{L_P}$ is generated by $7$ cubics and at most $7$ quartics.

Let us now show that $L_P$ needs at least $7$ quartic generators for generic $P$.
We use Singular to compute for a fixed $P_0$ that the scheme-theoretic intersection of cubics from
$I_{L_{P_0}}$ defines the reduced sum of the surface $L_{P_0}$ with seven $4$-secant lines that we denote by $S_{P_0}$.
We find also that in this case $h^1(\mathcal{I}_{S_{P_0}}(4))=0$ and $h^0(\mathcal{I}_{S_{P_0}}(4))=38$.
On the other hand from the sequence (\ref{q2}) we have that $h^0(\mathcal{I}_{L_{P_0}}(4))=45$ is the maximal possible value (i.e.~this is a generic value when $P_0$ is a variable).
In order to use semi-continuity arguments we consider the family
$$\mathcal{B}\supset S \rightarrow U \subset G(3,9),$$
where $\mathcal{B}$ is the tautological $\mathbb{P}^5$ bundle over the
subset $U\subset G(3,9)$ and $S\subset \mathcal{B}$ a subset such that $S_P\subset \mathbb{P}^5$ is
the fiber over $P$.
\begin{lemm}\label{lem3}The family $S\rightarrow U$ is flat.\end{lemm}
\begin{proof}
Let $Q_P$ be the union of the quadrisecant lines and let $Q\rightarrow U$
be the natural family. We shall show that the families $Q$ and $Q\cap L$
with fiber $Q_P\cap L_P$ are flat over an open subset. Consider the
Hilbert scheme $Hilb^4(L_P)\simeq Hilb^4(D_8)$ and the scheme
$Al^4(\mathbb{P}^5)$ of aligned points on $ \mathbb{P}^5$. It is
proved in \cite{LB} that for generic $P\in G(3,9)$ the intersection
$Al^4(\mathbb{P}^5)\cap Hilb^4(L_P) \subset Hilb^4(\mathbb{P}^5)$
has degree $7$ as computed as at the begin of the proof. Moreover, the family
$$ \mathcal{L}\rightarrow U,$$ where $\mathcal{L}$ is the
natural family obtained from the family $\mathcal{B} \supset L
\rightarrow U$ by taking $Hilb_c^4(.)$ (see \cite{LB}) of the
fibers, is smooth, so in particular flat. Consider the natural
smooth fiber bundle of Hilbert schemes of $4$ points in a fiber
$\mathcal{H}\rightarrow U$ obtained from $\mathcal{B}\rightarrow
U$ and its subbundle $\mathcal{A}\rightarrow U$ such that
$\mathcal{A}_P$ is equal to $Al^4(\mathbb{P}^5_P)$.

We have a natural embedding $f\colon\mathcal{A}\rightarrow
\mathcal{H}$. Consider the pull-back $f^*(L)$ of the flat family
$\mathcal{L}\rightarrow U$. From \cite[III Prop.~9.1A]{Ha} the
family $p\colon f^*(L)\rightarrow U $ is flat (this is exactly the
family $Q\cap L$). It remains to remark that in our chosen example
the fiber of $p$ is smooth. There exists an open $V\subset U$ such
that $p^{-1}(V)\rightarrow V$ is smooth, thus flat.

We conclude that $Q\cap L$ is a flat family. Moreover, using the
natural morphism $Al^4(\mathbb{P}^5)\rightarrow G(1,5)$ we deduce
that $Q$ is a flat family.

Since $L\rightarrow V$ is a flat family, we deduce
from the exact sequence
$$0\rightarrow\mathcal{I}_{Q\cap L} \rightarrow
\mathcal{O}_{L}\rightarrow \mathcal{O}_{Q\cap L}\rightarrow 0$$
that $\mathcal{I}_{Q\cap L}$ is a flat $\mathcal{O}_{V}$-module.
From the exact sequence $$0\rightarrow\mathcal{I}_{Q} \rightarrow
\mathcal{O}_{S}\rightarrow \mathcal{O}_{Q}\rightarrow 0,$$ we
deduce that $S\rightarrow V$ is flat.\end{proof}

We deduce that $h(P)=h^0(\mathcal{I}_{S_{P}}(4))$ is upper semi-continuous.
Since the cubics from the ideal of $L_P$ vanish on $S_P$ we deduce that there are at least $45-38=7$ quartic generators.

Let us prove the second part by concentrating on the case of the surface $L_P$. We computed with Singular that there exists a $P_0$ such that the ideal of $S_{P_0}$ is generated by the seven cubics from the ideal of $L_{P_0}$.
Since for a generic $P$ the surface $L_P$ have seven cubic generators that vanish on $S_P$ our result follows if we can use the semi-continuity results about the Betti numbers of $S_P$. To apply \cite{BG} we have to show that the Hilbert function $h(P)=h^0(\mathcal{O}_{\mathbb{P}^5}(n))-h^0(\mathcal{I}_{S_P}(n))$ is maximal for $P_0$.

First we saw that the number $h^0(\mathcal{I}_{S_{P_0}}(4))=38$ is minimal.
We compute also that $S_{P_0}$ is $5$-regular so by semi-continuity and the Mumford criterium $S_P$ is $5$-regular for a generic $P$.
Since the Hilbert function is equal to the Hilbert polynomial for $n\geq reg(S_P)$ (see \cite[Thm.~4.2]{E}) the proof is complete.
\end{proof}
\begin{rem} It can be shown that the Betti numbers of projections of del Pezzo surfaces to $\mathbb{P}^5$ change when we choose a special center of the smooth projection.
\end{rem}
\begin{prop}\label{k0} The surfaces $L_P$ and $\tilde{\tilde{D_7}}$ are embedded into nodal Calabi--Yau threefolds in
$\mathbb{P}^5$ which are complete intersections $X_1$ and $X_2$ of
two cubic with $52$ (resp.~$44$) nodes. The Picard group of small resolutions of $X_1$ and $X_2$ has rank $2$ and the Calabi--Yau threefold obtained after
the smoothing of the primitive contractions of $X_1$ and $X_2$ have Hodge numbers $(1,31)$ and $(1,38)$ respectively; we obtain the same invariants as the corresponding Tonoli examples.\end{prop}
\begin{proof} The complete intersections are nodal by Lemmas \ref{lem nodal} and \ref{es}.
To apply the theorem \cite[Thm.~2]{RS} we use Singular.
Let us concentrate on the case of $L_P$. First we compute that the intersection of five generic elements of the system $H^0(\mathcal{I}_{L_P}(3))$ cuts along $L_P$ and a finite number of points.
 This shows that the system $H^0(\mathcal{I}_{L_P}(3))$ is big. Then we compute the dimension of the locus where the matrix of partial derivatives of the map have smaller rank and find that the exceptional set in $\mathbb{P}^5$ of the map given by the seven cubics is of codimension $3$
\end{proof}
\begin{rem} Note that the intersection of three generic quadrics
containing $\tilde{D_7}$ (resp.~$\tilde{\tilde{D_8}}$) is a nodal
Fano threefold with $16$ (resp.~$20$) nodes.
\end{rem}
\begin{rem}\label{re}
It is an open problem to find the mirror families of the obtained
Calabi--Yau threefolds. The strategy would be to find a weak
Landau--Ginzburg model following the Batyrev approach: i.e to embed
the given Calabi--Yau threefold as a complete intersection in a Fano
manifold, then to degenerate the Fano to a toric $T$ with terminal
Gorenstein singularities, and finally find the appropriate Laurant
polynomial using the generators of the fans of $T$ (see \cite{Pr}).
\end{rem}
\section{Appendix}
Recall that the ideal of the del Pezzo surface $D_7$ can be
described by the $2$ by $2$ minors of a partially symmetric $4$ by
$5$ matrix. The following script gives the ideal of a projection of
$D_7$ into $\mathbb{P}^6(x,y,z,t,u,v,w)$. Then compute the number of nodes on a generic complete intersection of a cubics and two quadrics containing $\tilde{D_7}$ with a method that is much faster then using the Jacobian matrix.
\\ \\
\emph{
ring r1=101,(x,y,z,t,u,v,w),dp;\\
ring r=101,(x,y,z,t,u,v,w,p),dp;\\
LIB"random.lib";\\
matrix A2=randommat(5,5,maxideal(1),7);\\
matrix A1=A2+transpose(A2);\\
matrix A=submat(A1,2..4,1..4);\\
ideal jj=minor(A,2);\\
ideal j=eliminate(jj,p);\\
map f=r1,x,y,z,t,u,v,w;
setring r1;\\
ideal i=preimage(r,f,j);\\
ideal h=intersect(i,maxideal(2));\\
ideal ha=intersect(i,maxideal(3));\\
matrix C=randommat(1,1,ha,7);\\
ideal k=h[3],h[4],C[1,1],h[6];\\
ideal s=quotient(k,i);\\
ideal d=h[3],h[4],C[1,1],s[6];\\
ideal e=quotient(d,s);\\
ideal a=e,i;\\
degree(std(a));\\
ideal ra=radical(a);\\
degree(std(ra));\\
}\\
Recall that the ideal of the del Pezzo surface $D_8$ can be described by the $2$ by
$2$ minors of a symmetric $4$ by $4$ matrix. The following script gives the ideal of a
projection $L_P$ of $D_8$ into $\mathbb{P}^5(x, y, z, t, u, v)$.\\ \\
\emph{ring r=0,(x,y,z,t,u,v,w,p,q),dp;\\
LIB"random.lib";\\
matrix A2=randommat(4,4,maxideal(1),7);\\
matrix A1=A2+transpose(A2);\\
ideal jj=minor(A1,2);\\
ideal j=eliminate(jj,wpq);\\
minbase(j);}
\section{Appendix 1}\label{A1}

We present a list of all Calabi-Yau threefolds with Picard group of
rank $1$ known to the author. In the references column we show the
place where we can find more information about this Calabi--Yau
threefold. We put $?$ when the Calabi--Yau threefold is only
conjectured to exist. Moreover $B$ denotes an appropriated Fano
threefold.
\begin{longtable}
 {|c|c|c|c|c|c|c|c|}\hline
 $H^3$&$h^{1,1}$&$h^{1,2}$&$\chi$&$c_2\cdot H$&$dim|H|$&Description&Reference         \\\hline\hline       \endhead
            1 &1&61&-120&22 & 2&   $X_{6,6}\subset{P}(1,1,2,2,3,3)$                       &\cite{KT}   \\ \hline
            1 &1&145&-288&34 & 3&   $X_{10}\subset \mathbb{P}(1,1,1,2,5)$                  &            \\ \hline
            2 &1& 23&-44?&20 & 2&       ?                                                  &\cite{EV}   \\ \hline
            2 &1& 79&-156&32 & 3&   $X_{4,6}\subset\mathbb{P}(1,1,1,2,2,3)$                &\cite{KT}   \\ \hline
            2 &1&149&-296&44 & 4&   $X_8\subset\mathbb{P}(1,1,1,1,4)$                      &            \\ \hline
            3 &1&103&-204&42 & 4&   $X_6\subset\mathbb{P}(1,1,1,1,2)$                      &            \\ \hline
            4 &1& 73&-144&40 & 4&   $X_{4,4}\subset\mathbb{P}(1,1,1,1,2,2)$                &\cite{KT}   \\ \hline
            4 &1&129&-256&52 & 5&   $X_{2,6}\subset\mathbb{P}(1,1,1,1,1,3)$                &\cite{KT}   \\ \hline
            5 &1& 51&-100&38 & 4&      ?                                                   &\cite{EV}   \\ \hline
            5 &1&101&-200&50 & 5&   $X_5\subset\mathbb{P}^4$                               &\cite{COGP} \\ \hline
            5 &1&156&-310&62 & 6&         ?                                                &\cite{EV}   \\ \hline
            6 &1& 79&-156&48 & 5&   $X_{3,4}\subset \mathbb{P}(1,1,1,1,1,2)$               &\cite{KT}   \\ \hline 
            6 &1& 37&-72 &36 & 4&           ?                                              &\cite{EV}   \\ \hline
            7 &1& 61&-120&46 & 5&           ?                                              &\cite{EV}   \\ \hline
            7 &1& 79&-156&58 & 6&                                                          &\cite{K}   \\ \hline
            8 &1&  5&-8  &32 & 4&          ?                                               &\cite{EV}   \\ \hline
            8 &1& 89&-176&56 & 6&   $X_{2,4}\subset\mathbb{P}^5$                           &\cite{LT}   \\ \hline
            9 &1& 73&-144&54 & 6&   $X_{3,3}\subset\mathbb{P}^5$                           &\cite{LT}   \\ \hline
            10&1& 26&-50 &40 & 5&          ?                                               &\cite{EV}   \\ \hline
            10&1& 10&-32 &40 & 5&          ?                                               &\cite{EV}   \\ \hline
            10&1& 59&-116&52 & 6&           ?                                              &\cite{EV}   \\ \hline
            10&1& 59&-116&62 & 7&                                                          &\cite{K}   \\ \hline
            12&1& 10&-32 &36 & 6&          ?                                               &\cite{EV}   \\ \hline
            12&1& 31&-60 &48 & 6&           ?                                              &\cite{EV}   \\ \hline
            12&1& 35&-68 &48 & 6&   $X \xrightarrow{2:1}B$                                 &\cite{Lee1} \\ \hline
            12&1&61&-120&84 & 9&                                                          &\cite{K}   \\ \hline
            12&1& 73&-144&60 & 7&   $X_{2,2,3}\subset\mathbb{P}^6$                         &\cite{LT}   \\ \hline
            13&1& 61&-120&58 & 7&   $5\times 5 \ \mbox{Pffafian}\subset\mathbb{P}^6$       &\cite{T}\\ \hline
            13&1& 52&-102&82 & 9&                                                          &\cite{K}   \\ \hline
            14&1& 43&-84 &80 & 9&                                                          &\cite{K}   \\ \hline
            14&1& 44&-86 &80 & 9&                                                          &\cite{K}   \\ \hline
            14&1& 50&-98 &56 & 7&   $7\times 7 \ \mbox{Pffafian}\subset\mathbb{P}^6$       &\cite{Rodland} Table \ref{table1} \\ \hline
            14&1& 49&-96 &56 & 7&   $X \xrightarrow{2:1}B$                                 &\cite{Lee1} \\ \hline
            14&1& 51&-100&56 & 7&   $X \xrightarrow{2:1}B$                                 &\cite{Lee1} \\ \hline
            14&1& 51&-100&56 & 7&            ?                                             &\cite{EV}   \\ \hline
            14&1& 61&-120&68 & 8&                                                          &\cite{K}   \\ \hline
            15&1& 35&-68 &78 & 9&                                                          &\cite{K}   \\ \hline
            15&1& 39&-76 &54 & 7&    $X \xrightarrow{2:1}B$                                &\cite{Lee1} \\ \hline
            15&1& 39&-76 &54 & 7&                                                          &Table \ref{table1}   \\ \hline
            15&1& 40&-78 &54 & 7&                                                          &Table \ref{table1}   \\ \hline
            15&1& 40&-78 &54 & 7&   $To_{15}\subset\mathbb{P}^6$                           &\cite{T} \\ \hline
            15&1& 43&-84 &54 & 7&   $X \xrightarrow{2:1}B$                                 &\cite{Lee1} \\ \hline
            15&1& 76&-150&66 & 8&   $X_{1,1,3}\subset G(2,5)$                              &\cite{BCKS} \\ \hline
            16&1& 31&-60 &52 & 7&                                                          &Table \ref{table1}  \\ \hline
            16&1& 31&-60 &52 & 7&      $To_{16}\subset\mathbb{P}^6$                        &\cite{T}\\ \hline
            16&1& 37&-72 &52 & 7&   $X \xrightarrow{2:1}B$                                 &\cite{Lee1} \\ \hline
            16&1& 65&-128&64 & 8&   $X_{2,2,2,2}\subset\mathbb{P}^7$                       &\cite{LT}   \\ \hline
 $\frac{16}{n^3}$&1&65&-128 &  & 8&                                                       &\cite{K}   \\ \hline
            17&1& 23&-44& 50 & 7&                                                          &Table \ref{table1}   \\ \hline
            17&1& 23&-44 &50 & 7&   $To_{17}\subset\mathbb{P}^6$                           &\cite{T}\\ \hline
            17&1& 55&-108&62 & 8&                                                          &\cite{K}   \\ \hline
            17&1& 55&-108&62 & 8&   $X \xrightarrow{2:1}B$                                 &\cite{Lee1} \\ \hline
            17&1& 33&-64 &50 & 7&   $X \xrightarrow{2:1}B$                                 &\cite{Lee1} \\ \hline
            18&1& 45&-88 &60 & 8&                                                          &\cite{K}   \\ \hline
            18&1& 45&-88 &60 & 8&              ?                                           &\cite{EV}   \\ \hline
            18&1& 43&-84 &60 & 8&  $X \xrightarrow{2:1}B$                                  &\cite{Lee1} \\ \hline
            18&1& 46&-90 &60 & 8&                                                          &\cite{K}   \\ \hline
            18&1& 47&-92 &60 & 8&   $X \xrightarrow{2:1}B$                                 &\cite{Lee1} \\ \hline
       $18n^3$&1& 54&-106&   &  &                                                          &\cite{K}   \\ \hline
            19&1& 38&-74 &58 & 8&                                                          &Table \ref{table1}   \\ \hline
            19&1& 39&-76 &58 & 8&   $X \xrightarrow{2:1}B$                                 &\cite{Lee1} \\ \hline
            20&1& 31&-60 &56 & 8&                                                          &Table \ref{table1} \\ \hline
            20&1& 61&-120&68 & 9&                                                          &\cite{K}   \\ \hline
            20&1& 61&-120&68 & 9&   $X_{1,2,2}\subset G(2,5)$                              &\cite{BCKS} \\ \hline
            21&1& 52&-102&66 & 9&               ?                                          &\cite{EV}   \\ \hline
            21&1& 51&-100&66 & 9&   $X \xrightarrow{2:1}B$                                 &\cite{Lee1} \\ \hline
            21&1& 53&-104&66 & 9&   $X \xrightarrow{2:1}B$                                 &\cite{Lee1} \\ \hline
            21&1& 51&-100&66 & 9&               ?                                          &\cite{EV}   \\ \hline
            22&1& 47&-92 &64 & 9&   $X \xrightarrow{2:1}B$                                 &\cite{Lee1} \\ \hline
            24&1& 26&-50 &60 & 9&                                                          &Table \ref{table1}   \\ \hline
            24&1& 59&-116&72 &10&  $X_{1,1,1,1,1,1,2}\subset X_{10}$                       &\cite{EV}            \\ \hline
            25&1& 51&-100&70 &10&                                                          &\cite{K}   \\ \hline
            25&1& 51&-100&70 &10&            ?                                             &\cite{EV}   \\ \hline
            25&1& 51&-100&70 &10&  $X \xrightarrow{2:1}B$                                  &\cite{Lee1} \\ \hline
            28&1& 59&-116&76 &11&  $X_{1,1,1,1,2}\subset G(2,6)$                           &\cite{EV}            \\ \hline
            29&1& 51&-100&74 &11&             ?                                            &\cite{EV}   \\ \hline
            29&1& 53&-104&74 &11& $X \xrightarrow{2:1}B$                                   &\cite{Lee1} \\ \hline
            29&1& 49&-96 &74 &11&  $X \xrightarrow{2:1}B$                                  &\cite{Lee1} \\ \hline
            30&1& 49&-96 &72 &11&  $X \xrightarrow{2:1}B$                                  &\cite{Lee1} \\ \hline
            32&1& 59&-116&80 &12&  $X_{1,1,2}\subset LG(3,6)$                              &\cite{EV}            \\ \hline
            32&1& 59&-116&80 &12&                                                          &\cite{BK}   \\ \hline
            33&1& 52&-102&78 &12&             ?                                            &\cite{EV}   \\ \hline
            34&1& 45&-88 &76 &12&              ?                                           &\cite{EV}   \\ \hline
            34&1& 49&-96 &76 &12&                                                          &\cite{K}   \\ \hline
            34&1& 50&-98 &76 &12&                                                          &\cite{K}   \\ \hline
            35&1&26&-50&?&11&$3\times 3$ minors of $5\times 5$ sym. mat.&\cite{KK1}\\ \hline
            36&1& 37&-72 &72 &12&              ?                                           &\cite{EV}   \\ \hline
            36&1& 61&-120&84 &13&  $X_{1,2}\subset X_5$                                    &\cite{EV}            \\ \hline
            42&1& 50&-98 &84 &14&  $X_{1,1,1,1,1,1,1}\subset G(2,7)$                       &\cite{BCKS} \\ \hline
            42&1& 49&-96 &84 &14&  $X_{1,1,1,1,1,1}\subset G(3,6)$                         &\cite{BCKS} \\ \hline
            44&1& 65&-128&92 &15&  $X_{2,1}\xrightarrow{2:1}A_{2,2}$                       &\cite{EV}            \\ \hline
            47&1& 46&-90 &86 &15&              ?                                           &\cite{EV}   \\ \hline
            48&1& 79&-156&96 &16&                                                          &\cite{BK}   \\ \hline
            56&1& 47&-92 &92 &17&  $X_{1,1,1,1}\subset F_1(Q_5)$                           &\cite{EV}            \\ \hline
            57&1& 43&-84 &90 &17&  Tj                                                      &\cite{EV}            \\ \hline
            74&1& 29&-56 &92 &20&                                                          &\cite{BK}   \\ \hline
            78&1& 31&-60 &96 &21&                                                          &\cite{BK}   \\ \hline
            78&1& 33&-64 &96 &21&                                                          &\cite{BK}   \\ \hline
            79&1& 25&-48 &94 &21&                                                          &\cite{BK}   \\ \hline
            80&1&101&-200&128&24&                                                          &\cite{BK}   \\ \hline
            82&1& 36&-70 &100&22&                                                          &\cite{BK}   \\ \hline
            83&1& 31&-60 &98 &22&                                                          &\cite{BK}   \\ \hline
            83&1& 32&-62 &98 &22&                                                          &\cite{BK}   \\ \hline
            86&1& 41&-80 &104&23&                                                          &\cite{BK}   \\ \hline
            87&1& 35&-68 &102&23&                                                          &\cite{BK}   \\ \hline
            88&1& 29&-56 &100&23&                                                          &\cite{BK}   \\ \hline
            91&1& 40&-78 &106&24&                                                          &\cite{BK}   \\ \hline
            92&1& 35&-68 &104&24&                                                          &\cite{BK}   \\ \hline
            92&1& 36&-70 &104&24&                                                          &\cite{BK}   \\ \hline
            93&1& 29&-56 &102&24&                                                          &\cite{BK}   \\ \hline
            96&1& 39&-76 &108&25&                                                          &\cite{BK}   \\ \hline
            97&1& 33&-64 &106&25&                                                          &\cite{BK}   \\ \hline
            97&1& 34&-66 &106&25&                                                          &\cite{BK}   \\ \hline
            97&1& 35&-68 &106&25&                                                          &\cite{BK}   \\ \hline
            98&1& 29&-56 &104&25&                                                          &\cite{BK}   \\ \hline
            98&1& 30&-58 &104&25&                                                          &\cite{BK}   \\ \hline
            98&1& 31&-60 &104&25&                                                          &\cite{BK}   \\ \hline
            98&1& 32&-62 &104&25&                                                          &\cite{BK}   \\ \hline
            99&1& 28&-54 &102&25&                                                          &\cite{BK}   \\ \hline
           102&1& 34&-66 &108&26&                                                          &\cite{BK}   \\ \hline
           102&1& 35&-68 &108&26&                                                          &\cite{BK}   \\ \hline
           102&1& 38&-74 &108&26&                                                          &\cite{BK}   \\ \hline
           103&1& 30&-58 &106&26&                                                          &\cite{BK}   \\ \hline
           103&1& 31&-60 &106&26&                                                          &\cite{BK}   \\ \hline
           104&1& 28&-54 &104&26&                                                          &\cite{BK}   \\ \hline
           107&1& 36&-70 &110&27&                                                          &\cite{BK}   \\ \hline
           108&1& 30&-58 &108&27&                                                          &\cite{BK}   \\ \hline
           108&1& 31&-60 &108&27&                                                          &\cite{BK}   \\ \hline
           108&1& 32&-62 &108&27&                                                          &\cite{BK}   \\ \hline
           108&1& 33&-64 &108&27&                                                          &\cite{BK}   \\ \hline
           108&1&129&-256&156&31&                                                          &\cite{BK}   \\ \hline
           112&1& 35&-68 &112&28&                                                          &\cite{BK}   \\ \hline
           113&1& 32&-62 &110&28&                                                          &\cite{BK}   \\ \hline
           116&1& 41&-80 &116&29&                                                          &\cite{BK}   \\ \hline
           117&1& 37&-72 &114&29&                                                          &\cite{BK}   \\ \hline
           118&1& 31&-60 &112&29&                                                          &\cite{BK}   \\ \hline
           118&1& 32&-62 &112&29&                                                          &\cite{BK}   \\ \hline
           123&1& 34&-66 &114&30&                                                          &\cite{BK}   \\ \hline
           124&1& 31&-60 &112&30&                                                          &\cite{BK}   \\ \hline
           136&1& 55&-108&124&33&                                                          &\cite{BK}   \\ \hline
           144&1& 45&-88 &120&34&                                                          &\cite{BK}   \\ \hline
           144&1& 47&-92 &120&34&                                                          &\cite{BK}   \\ \hline
           152&1& 40&-76 &116&35&                                                          &\cite{BK}   \\ \hline
           168&1& 51&-100&132&39&                                                          &\cite{BK}   \\ \hline
           168&1& 53&-104&132&39&                                                          &\cite{BK}   \\ \hline
           176&1& 47&-92 &128&40&                                                          &\cite{BK}   \\ \hline
           200&1& 51&-100&140&45&                                                          &\cite{BK}   \\ \hline
           232&1& 53&-104&148&51&                                                          &\cite{BK}   \\ \hline
           432&1& 79&-156&192&88&                                                          &\cite{BK}   \\ \hline
           648&1&103&-204&252&129&                                                         &\cite{BK}   \\ \hline
?&1&1&?&?&?&    not simply connected &\cite{D}\\ \hline ?&1&4&?&?&?&
not simply connected &\cite{D}\\ \hline

\end{longtable}

\vskip10pt Department of Mathematics and Informatics,\\ Jagiellonian
University, {\L}ojasiewicza 6, 30-348 Krak\'{o}w, Poland.\\
 \emph{E-mail address:} grzegorz.kapustka@uj.edu.pl
\end{document}